\documentclass[11pt,a4paper,oneside,reqno,notitlepage]{amsart}
\usepackage{amssymb,amsmath}

\usepackage{latexsym, amsbsy, amsthm, amscd}
\usepackage{amsfonts}
\usepackage{amscd}
\usepackage{euscript}

\newtheorem{cor}{Corollary}

\newtheorem{prop}{Proposition}
\newtheorem{theorem}{Theorem}
\newtheorem{lemma}{Lemma}
\def\Wd{\operatorname{Wd}}
\def\Id{\operatorname{Id}}

\begin{document}
\title{Free subalgebras of Lie algebras close to nilpotent}
\author{Alexey Belov and Roman Mikhailov}

\begin{abstract}
We prove that for every automata algebra of exponential growth, the
associated Lie algebra contains a free subalgebra. For $n\geq 1$,
let $L_{n+2}$ be a Lie algebra with generator set $x_1,\dots,
x_{n+2}$ and the following relations: for $k\leq n$, any commutator
of length $k$ which consists of fewer than $k$ different symbols
from $\{x_1,\dots,x_{n+2}\}$ is zero. As an application of this
result about automata algebras, we prove that for every $n\geq 1$,
$L_{n+2}$ contains a free subalgebra. We also prove the similar
result about groups defined by commutator relations. Let $G_{n+2}$
be a group with $n+2$ generators $y_1,\dots, y_{n+2}$ and the
following relations: for $k=3,\dots, n$, any left-normalized
commutator of length $k$ which consists of fewer than $k$ different
symbols from $\{y_1,\dots,y_{n+2}\}$ is trivial. Then the group
$G_{n+2}$ contains a $2$-generated free subgroup.

Main technical tool is combinatorics of periodical sequences and
period switching.
\end{abstract}

\maketitle

\section{Introduction}
Let $A$ be an associative algebra over a commutative ring with
identity, generated by a set $S$. Denote by $A^{\sim}$ the Lie
algebra with the same set of generators $S$ and operation
$[u,v]=uv-vu,\ u,v\in A.$ In other words $A^\sim$ is the Lie
subalgebra of $A^-$ generated by the given set $S$. The algebra
$A^{\sim}$ clearly depends on the choice of the set of generators of
$A$.

For $n\geq 1$, let $L_{n+2}$ be a Lie algebra with generator set
$x_1,\dots, x_{n+2}$ and following relations: for $k\leq n$, any
commutator of length $k$ which consists of fewer than $k$
different symbols from $\{x_1,\dots,x_{n+2}\}$ is zero. For
example, the trivial commutators in $L_n$, which correspond to the
case $k=3$, are:
$$
[[x_i,x_j],x_i],\ i\neq j.
$$
One of the main results of this paper is following:

\begin{theorem}\label{lie}
For every $n\geq 1,$ $L_{n+2}$ contains a free Lie subalgebra.
\end{theorem}

The proof of theorem \ref{lie} is based on the theory of monomial
algebras. An algebra with basis $X$ is called {\it monomial} if
all its defining relations are of the form $u=0$, where $u$ is a
word written in $X$. Let $A$ be a finitely generated algebra with
generators $x_1,\dots,x_s$. The {\it Growth function} $V_A(n)$
equal, the dimension of the space generated by words of length
$\le n$. If $V_A(n)$ grows exponentially, then $A$ has {\it
exponential growth}; if polynomially, then $A$ has {\it polynomial
growth}. {\it Intermediate growth} is also possible. The
polynomial or exponential growth property does not depend on the
choice of the generator set.

For $n\geq 1$, let $A_{n+2}$ be the monomial algebra with
generators $x_1,\dots, x_{n+2}$ and the following relations:
$u(x_1,\dots,x_{n+2})=0$ if $|u|=k$ $(k\leq n)$ and $u$ consists
of fewer than $k$ symbols from $\{x_1,\dots, x_{n+2}\}.$ Clearly,
the Lie algebra $A_{n+2}^{\sim}$ is a quotient of $L_{n+2}$.
Hence, Theorem \ref{lie} will follow if we will be able to prove
that $A_{n+2}^{\sim}$ contains a free subalgebra. The algebra
$A_{n+2}$ has an alternative description, based on the following
property: $u(x_1,\dots,x_{n+2})=0$ in the algebra $A_{n+2}$ if the
distance between two occurrences of the same letter in
$u(x_1,\dots,x_{n+2})$ is less than $n+1$.

Consider a super-word $w=(x_1\cdots x_{n+1})^\infty$. It is clear
that $w\neq 0$ and any series of changes $x_{n+1}\mapsto x_{n+2}$
will not yield zero, since the distance between two occurrences of
the same letter is still $\leq n$.

With the help of above changes it is possible to get $2^M$
different non--ze\-ro words. It follows that the number of
different non--ze\-ro words of length $k$ in the monomial algebra
$A$ is not less than $2^{[\frac{k}{n+1}]}$, and the algebra $A$
has exponential growth. Algebra $A_{n+2}$ is
fi\-ni\-te\-ly--presented and hence is an automata \cite{BBL}.
Hence $A_{n+2}^{\sim}$ contains a $2$-generated (and countably
generated) free subalgebra, by Theorem \ref{main}.

\medskip
\noindent{\bf Remark.} More complicated proof is needed for the fact
that $A_{n+1}^{\sim}$ also contains a free subalgebra for every
$n\geq 3$.
\medskip

The similar situation takes place in the case of groups. As a
natural group-theoretical analog of Theorem \ref{lie}, we have the
following:

\begin{theorem}\label{group}
For $n\geq 1$, let $G_{n+2}$ be a group with $n+2$ generators
$y_1,\dots, y_{n+2}$ and the following relations: for $k=3,\dots,
n$, any left-normalized commutator of length $k$ which consists of
fewer than $k$ different symbols from $\{y_1,\dots,y_{n+2}\}$ is
trivial. Then the group $G_{n+2}$ contains a $2$-generated free
subgroup.
\end{theorem}

Note that the groups $G_{n+2}$ are related to the construction of
amenable groups with fast F{\o}lner function from \cite{G}.
\vspace{.5cm}

\thanks It is a pleasure for us to thank Mikhail Gromov
for posing the problem and his helpful comments and Louis Rowen
for useful discussions and suggestions. The authors were supported
by the Israel Science Foundation grant No. 1178/06. The research
of the second author was partially supported by Russian Science
Support Foundation.

\section{The periodicity}

The order $a_1\prec a_2\prec\dots\prec a_n$ induces a
lexicographical order on the set of all words. Two words are
incompatible with respect to this order, if one is initial in the
of other. By $|v|$ will be denoted the length of a word $v$. By
$u\subset v$ will be denoted the occurrence of a word $u$ in a
word $v$.
 A word $u$ is called {\it cyclic}, if for some $k>1$,
$u=v^k;$ otherwise it is called {\it noncyclic} or {\it
nonperiodic}.  If $W=u^k r$, where $r$ is an initial segment in
$u$, then $W$ is called {\it quasiperiodic of order $|u|$}. In
this case $W$ is a subword of $u^\infty$ (see next paragraph).
Words $u$ and $v$ are called {\it cyclically conjugate}, if, for
some words $c$ and $d$, $u=cd$ and $v=dc$. The {\it cyclic
conjugacy} relation is an equivalence relation.

A {\it superword} is a word which is infinite in both directions.
A word which is infinite to the left, is called a {\it left
superword}; a word which is infinite to the right, is called a
{\it right superword}. By $u^\infty$ will be denoted a superword
with the period $u$, by $u^{\infty/2}$ will be denoted a right
(left) superword, which begins (terminates) with the word $u$.

Since it will clear from the context, which superword is under
consideration, left or right, we do not introduce special
notation. The notation $u^{\infty/2}\cdot s\cdot v^{\infty/2}$
means, for example, that $u^{\infty/2}$ is a left superword and
$v^{\infty/2}$ is a right one.

Right superwords (unlike finite words, for which incomparable
elements exist) constitute a linear ordered set with respect to
the left lexicographic ordering; the same is true for left
superwords with respect to the right lexicographic ordering.

\subsection{Periodic superwords}
{\bf Subwords of $u^\infty$.}\ By $u$ will be denoted a
non-cyclical word. We recall some propositions from \cite{BBL}. By
$A_{u^\infty}$ we denote an algebra whose defining relations has
following form: $s=0$ where $s$ is a word which is not a subword
of $u^\infty$.

\begin{prop}             
Each two subwords in $u^\infty$ of length $N|u|$ are cyclically
conjugate; they coincide only when the distance between their
first letters is divisible by the period.
\end{prop}

\begin{prop}             \label{Th2.3}
a) The beginning subword of length $|u|-1$ uniquely defines the
word from $A_{u^\infty}$. If the initial subwords of length
$|u|-1$ in two subwords $v$ and $v'$ coincide ($v$ and $v'$ are
subwords in the superword $u^\infty$), then one of them is a
subword in another. If $|c|\ge|u|$ and $d_1$ and $d_2$ are
lexicographically comparable, then at least one of the words
$cd_1,cd_2$ is not a subword in $u^\infty$.

b) The positions of the occurrences of a word $v$ of length
$\ge|u|$ in $u^\infty$ differs by a period multiple.

c) If $|v|\ge |u|$ and $v^2\subset u^\infty$, then $v$ is
cyclically conjugate to a power of $u$. Therefore, nonnilpotent
words in $A_{u^\infty}$ are exactly those words, which are
cyclically conjugate to words of the form $u^k$.
\end{prop}

\begin{prop}             \label{Th2.7}
If $uW=Wr$, then $uW$ is a subword of $u^\infty$ and $W=u^n r$,
where $r$ is an initial segment in $u$.
\end{prop}

\noindent{\bf Remark.} The periodicity of an infinite word means its
invariance with respect to a shift. In the one-sided infinite case a
pre-period appears; in the finite case there appear effects related
to the truncation. This, together with superword technique is the
essence of a great many combinatorial arguments (see Proposition
\ref{Th2.7}) especially Bernside type problems. Proofs of the
Shestakov hypothesis (nilpotency of subalgebra of $n\times n$ matrix
algebra with all words of length $\le n$ are nilpotent), of the
Shirshov height theorem (the normal basis of associative affine
$PI$-algebra $A$ contains only piece-wise periodic words, number of
periodic parts is less then $h(A)$, and length of each period is
$\le n$~-- maximal dimension of matrix algebra satisfying all
identities of $A$) of the coincidence theorem of the nilradical and
the Jacobson radical in a monomial algebra, are examples \cite{BBL},
\cite{AmitsurSmall0}, \cite{Belov10}.
\medskip

\begin{lemma}[on overlapping]          \label{Lh2.15}
If a subword of length $m+n-1$ occurs simultaneously in two periodic
words of periods $m$ and $n$, then they are the same, up to a shift.
\end{lemma}

Lemma \ref{Lh2.15} implies one technical statement, needed in sequel

\begin{lemma}  \label{LePerinR}
Let $r=u^nv^m$, $n>k, m>l$. Then $r$ has not common subwords with
$u^\infty$ of length $\ge n|u|+|u|+|v|-1=(n+1)|u|+|v|-1$ and has not
common subwords with $v^\infty$ of length $\ge (m+1)|v|+|u|-1$.
\end{lemma}

\subsection{Periods switching}

\begin{prop}\label{pr}
Let $u,v$ are not powers of the same word, $l|v|>2|u|$ and
$s|u|>2|v|$. Then $v^lu^s$ is not a subword of $u^\infty$ or
$v^\infty$.
\end{prop}

This proposition follows from the following assertion which is
follows at once from Lemma 1 (see \cite{BBL}): if two periodical
superwords of periods $m$ and $n$ have the common part of length
$>m+n-2$, then these words are identical. In this case, $u^sv^l$
is a subword $v^\infty$ and $u^\infty$.

\begin{prop} \label{pra}
Let $u,v$ are not powers of the same word, $l|v|>2|u|$ and
$s|u|>2|v|$. Then $v^lu^s$ is not a proper power.
\end{prop}

\begin{proof}
Without loss of generality we may suppose that both $u$ and $v$ are
non-cyclic. Suppose that $k>1$ and $z^k=v^lu^s$ for some non-cyclic
word $z$. Without loss of generality we can assume that $|v^l|\ge
|u^s|$.

If $k\ge 4$, then $v^l$ contains $z^2$ and from the other hand,
$v^l$ is subword of $z^\infty$. It follows from overlapping lemma
\ref{Lh2.15} that $v$ is power of $z$ and hence $v=z$ (both $z, v$
are non cyclic). Then $u$ is also power of $z$ and we are done.

If $k=2$ then $u^s$ is a subword of $z$ and hence of $v^l$. That
contradicts overlapping lemma \ref{Lh2.15}.

If $k=3$ then because $|v^l|\ge |u^s|$ and $s|u|>2|v|$ we have $l\ge
3$. In this case $|v|<|z|/2$ and $|v^l|\ge |v|+|z|$. By overlapping
lemma \ref{Lh2.15} we have that $v^\infty=z^\infty$ and hence $v=z$
because both are non cyclic. Then $z^3=v^lu^s=z^lu^s$ and
$u^{s}=z^{3-l}$. Because $u$ is noncyclic $u=z$. Hence $u=v$ that
contradicts conditions of the proposition \ref{pra}.
\end{proof}

\begin{lemma}  \label{LePerinRa}
Let $r=u^nv^m$, $n>k, m>l$. Then
 $r$ is not a subword of $W'=v^{\infty/2}u^{\infty/2}$ and hence of
$v^pu^q$ for all $p, q$.
\end{lemma}

\begin{proof}
If $r$ is a subword of $W'$ then either $u^n$ (i.e. left part of
$r$) is a subword of $v^\infty$ or $v^m$ (i.e. right part of $r$) is
a subword of $u^\infty$. Both cases are excluded by proposition
\ref{pr}
\end{proof}

\begin{prop}      \label{Pruinftyvinfty}
Consider superword $W=u^{\infty/2}v^{\infty/2}$, where $u\ne v$
are different noncyclic words. Let $S=u^kv^l$ and suppose
$|u^{k-1}|>2|v|$, $|v^{k-1}|>2|u|$, $k, l\ge 2$.

Then $S$ has just one occurrence in $W$, which is the obvious one
(which we call the ``standard occurrence'').
\end{prop}

\begin{proof}
Otherwise the extra occurrence of $S$ is either to the left of the
standard occurrence, or to the right. Without loss of generality
it is enough to consider the left case.

In this case, by Proposition \ref{Th2.3} $W$ is shifted respect to
the standard occurrence by a distance divisible by $|u|$.

Hence we have: $u^sW=WR$, i.e. $u^sW$ starts with $W$. We can
apply Proposition \ref{Th2.7} and so we get that $u^{\infty/2}$
starts with $W$. Then from combining Lemma \ref{Lh2.15} and
Proposition \ref{Th2.3} we get that $v$ is cyclically conjugate to
$u$, and $|u|=|v|$.

But in that case $W=u^kv^l$ is subword of $u^\infty$, implying that
the relative shifts of $u$ and $v$ are divisible by $|u|=|v|$, and
hence $u=v$. The proposition is proved.
\end{proof}

This proposition together with Lemmas \ref{LePerinR} and
\ref{LePerinRa} implies

\begin{cor}  \label{CoShiftPower}
Let $R=r^\infty=(u^nv^m)^\infty$, $n>k, m>l$. Then all the
occurrences of $S$ in $R$ are separated by distances divisible by
$|r|=n|u|+m|v|$.
\end{cor}

\begin{proof}
First of all, as in the proposition \ref{Pruinftyvinfty} one can
define notion of {\it standard occurrence} of $S$ in $R$. Consider
an occurrence of $S$ in $R$. Then only following cases are logically
possible:

\begin{enumerate}
\item It naturally corresponds to occurrence of $S$ in $W$ (i.e. power of
$v$ in $S$ starts in on the power of $v$ in $W$ and similarly power
of $u$ in $S$ ends in on the power of $u$ in $W$).
\item It contains completely either $u^n$ or $v^m$.
\item It lies on the position of period switching from $v^m$ to $u^m$.
\end{enumerate}

Second possibility is excluded due to due to Lemma \ref{LePerinR},
third -- to due to Lemma \ref{LePerinRa}. First possibility due to
proposition \ref{Pruinftyvinfty} corresponds only to standard
occurrences and they are separated by distances divisible by
$|r|=|u^nv^m|=n|u|+m|v|$.
\end{proof}

Note that $r$ and $t$ are cyclically conjugate, iff
$r^\infty=t^\infty$. Using this observation and the previous
corollary we get a proposition needed in the sequel:

\begin{prop}
Let $u, v$ be different non-cyclic words, with  $|u^n|>2|v|$ and
$|v^n|>2|u|$. For all $k_i, l_i\ge n,$ if $(k_1,l_1)\ne
(k_2,l_2),$ then $u^{k_1}v^{l_1}$ and $u^{k_2}v^{l_2}$ are not
cyclically conjugate.
\end{prop}

\begin{proof}
Suppose that $r=u^{k_1}v^{l_1}$ and $t=u^{k_2}v^{l_2}$ are
cyclically conjugate. Then $r^\infty=t^\infty$ and because $r, t$
are not cyclical, $|r|=|t|=\mu$. Let us denote
$R=(u^{k_2}v^{l_2})^\infty$. Let $S=u^kv^l$ and suppose
$|u^{k-1}|>2|v|$, $|v^{k-1}|>2|u|$, $k, l\ge 2$ and also
$k\le\min{k_1,k_2}, l\le\min(l_1,l_2)$. It is clear that such $S$
exist and is a subword of booth $r$ and $t$.

Then due to corollary \ref{CoShiftPower} all occurrences of $S$ in
$W$ are shifted by distance divisible by $\mu$~-- period of $R$. It
means that any occurrence of $S$ can be extended to occurrence of
$r$ as well as to occurrence of $t$. Hence there exists an
occurrences of $r=u^{n_1}Sv^{m_1}$ and $t=u^{n_2}Sv^{m_2}$ in $W$
with common part $S$.

If $r\ne t$, then $n_1\ne n_2$ because $|s|=|t|$. Without loss of
generality we can suppose that $n_1<n_2$. In this case $m_1>m_2$.
Word $t$ is shifted to the left from the word $r$ on the distance
$d=|u^{n_1}|-|u^{n_2}|=|v^{n_2}|-|v^{m_1}|$.

Consider a union $\omega$ of $r$ and $t$. Then $\omega=er=ft$,
$|e|=|f|=d$. Note that $r^2$ is a subword of $t^\infty=W$,
occurrence of $v^{m_1}$ (which is end of $r$) precedes an occurrence
of $r$. Because $|v^{m_2}|>d$, $e=v^{m_2-m_1}$. Similarly
$f=u^{n_1-n_2}$.

From other hand $\omega$ can be also obtained by extending the
subword $S$ of $W$ to the left on the distance $|v^{\max(m_1,m_2)}|$
and to the right on the distance $|u^{\max(n_1,n_2)}|$ and
$\omega=u^{\max(n_1,n_2)}v^{\max(m_1,m_2)}=u^{n_1-n_2}r=tv^{m_2-m_1}$.

Hence $u^{n_1-n_2}=v^{m_2-m_1}$; $n_1\ne n_2; m_1\ne m_2$. It
follows that $u, v$ are powers of the same word $s$. Because $u\ne
v$ one of this powers is greater than $1$ and booth $u$ and $v$ can
not be non cyclic words. But this contradicts to their initial
choice.
\end{proof}

\section{Regular Words and Lie brackets}
We shall extend the relation $\prec$ by
defining the following $\rhd$-relation (``Ufnarovsky order''):
$f\rhd g$, if, for any two right superwords $W_1,W_2$, such that
$W_2(a,b)\succ W_1(a,b)$, when ever $b\succ a$, the inequality
$W_2(g,f)\succ W_1(g,f)$ holds. This condition is well defined and
equivalent to following: $f\rhd g$ iff $f^{\infty/2}\succ
g^{\infty/2}$ (i.e., $f^m\succ g^n$, for some $m$ and $n$). It is
clear that if $f\succ g$, then $f\rhd g$.

The relation $\rhd$ is a linear ordering on the following set of
equivalence classes: $f\sim g$, if for some $s$, $f=s^l$, $g=s^k$.

 Let us note that each finite word $u$ uniquely corresponds
to the right superword $u^\infty$. To equivalent words correspond
the same superwords. The relation $\rhd$ corresponds to the relation
$\succ$ on the set of superwords.

It is known (\cite{BBL}, \cite{Ufn}) that:

A word $u$ is called {\it regular}, if one of the following
equivalent conditions holds:

a) $u$ word is greater all its cyclic conjugates: If $u_1u_2=u$,
then $u\succ u_2u_1$.

b) If $u_1u_2=u$, then $u\rhd u_2$.

c) If $u_1u_2=u$, then $u_1\rhd u$.

A word $u$ is called {\it semi-regular} in the following case: If
$u=u_1u_2$, then, either $u\succ u_2$, or $u_2$ is a beginning of
$u$. (An equivalent definition can be obtained if the relation
$\lhd$ is replaced by the relation $\unlhd$ in the definition of a
regular word.) Every semi-regular word is a power of a regular
one.

It is well-known that every regular word $u$ defines the unique
bracket arrangement $[u]$ such that after opening all Lie brackets
$u$ will be a highest term in this expression. Moreover, monomials
of such type form a basis in the free Lie algebra (so called {\it
Hall--Shirshov basis}) ( see \cite{Bahturin}, \cite{Ufn}).

We shall need some technical statements:

\begin{lemma}[\cite{BBL}]  \label{LeConjswitch}
Suppose $|u^k|\rhd |v^2|$ and $u^k$ is a subword of $v^\infty$.
Then there exists $S'$ cyclically conjugate to $S$, such that
$u=(S')^m$ and $v=(S^2)^n$. If, moreover, the initial symbols of
$u$ and $v$ are at a distance divisible by $|S|$ in $v^\infty$,
then $S=S'.$
\end{lemma}

\begin{cor}    \label{Co2SrlrPwrs}
Let $u\rhd v$ be semi-regular words. Then, for sufficiently large
$k$ and $l$, the words $u^kv^l$ are regular and
$$
u^{k_1}v^{l_1}\rhd u^{k_2}v^{l_2}
$$
for $k_1>k_2$.
\end{cor}

\begin{proof}
Let $\delta$ be a cyclic conjugate of $u^kv^l$. It is clear that
$\delta \unrhd u^kv^l$, we only need to prove inequality $\delta \ne
u^kv^l$. In order to do this, we need only to show that $u^kv^l$ is
not cyclic word, but it follows from the proposition \ref{pra}.

\end{proof}

The next lemma follows from Lemma \ref{LeConjswitch} and Corollary
\ref{Co2SrlrPwrs}.

\begin{lemma}     \label{Le2powersConj}
Let $k_i>|d|,\ l_i>|u|,$ for $i=1,2$. Then $u^{k_1}d^{l_1}$ and
$u^{k_2}d^{l_2}$ are not cyclically conjugate, provided that
$u\rhd d$ and $u,d$ are not conjugate to proper powers of the same
word.
\end{lemma}

\section{Words in automata algebras}
By $\Phi\langle x_1,\!\dots\!,x_s\rangle$ will be denoted the free
associative $\Phi$-al\-geb\-ra with generators $x_1,\!\dots\!,x_s$.
By $A\langle a_1,\dots,a_s\rangle$ will be denoted an arbitrary
$\Phi$-al\-geb\-ra with a fixed set of generators $a_1,\dots, a_s$.
A {\it word} or a {\it monomial} from the set of generators
$\mathcal{M}$ is an arbitrary product of elements in $\mathcal M$.
The set of all words constitutes a semigroup, which will be denoted
by $\Wd\langle \mathcal M\rangle$. The order $a_1\prec \cdots\prec
a_s$ generates the lexicographic order on the set of words: The
grater of two words  is the one whose first letter is greater; if
the first symbols coincide, then the second letter are compared,
then the third letters and so on. Two words are incomparable, only
if one of them is initial in the other.

By a {\it word} in an algebra we understand a nonzero word from
its generators $\{a_i\}$. We cannot speak about the {\it value} of
a superword in an algebra, but can speak about its {\it equality}
or {\it nonequality to zero} (and, in some cases, about {\it
linear dependence}). A superword $W$ is called {\it zero
superword}, if it has a finite zero subword, and it is called a
{\it nonzero superword}, if it has no finite zero subwords.

An algebra $A$ is called {\it monomial}, if it has a base of
defining relations of the type $c=0$, where $c$ is a word from
$a_1,\dots,a_s$. Obviously, a monomial algebra is a semigroup
algebra (it coincides with the semigroup algebra over the semigroup
of its words).

\subsection{Automata algebras}

 First we  recall some well known definitions from \cite{BBL}.
Suppose we are given an alphabet (i.e., a finite set) $X$. {\it By
finite automaton (FA)} with the alphabet $X$ of input symbols we
shall understand an oriented graph $G$, whose edges are marked
with the letters from $X$. One of the vertices of this graph is
marked as initial, and some vertices are marked as final. A word
$w$ in the alphabet $X$ is called {\it accepted} by a finite
automaton, if there exists a path in the graph, which begins at
the initial vertex and finishes in some final vertex, such that
marks on the path edges in the order of passage constitute the
word $w$.

By a {\it language} in the alphabet $X$ we understand some subset
in the set of all words (chains) in $X$. A language $L$ is called
regular or automata, if there exists a finite automaton which
accepts all words from $L$ and only them.

An automaton is called {\it deterministic}, if all edges, which
start from one vertex are marked by different letters (and there
are no edges, marked by the empty chain). If we reject such
restriction and also allow edges, marked by the empty chain, then
we shall come to the notion of a {\it non-deterministic} finite
automaton. Also we can allow an automaton to have several initial
vertices. The following result from the theory of finite automata
is well known:

{\it For each non-deterministic FA there exists a deterministic
FA, which accepts the same set of words (i.e. the same language).}

It will be convenient for us to consider the class of FA, such that
all vertices are initial and final simultaneously. The reason of
this is that the language of nonzero words in a monomial algebra has
the following property: each subword of a word belonging to the
language, also belongs to it.

Suppose throughout that $G$ is the graph of a deterministic FA,
$v$ is a vertex of $G$, and $w$ is a word. If the corresponding
path $C$ starting from $v$ exist in $G$, then one can define the
vertex $vw$ terminal vertex for $C$.

Let $A$ be a monomial algebra (not necessary finitely defined). $A$
is called an {\it automata algebra}, if the set of all of its
nonzero words from $A$ generators is a regular language. Obviously,
a monomial algebra is an automata algebra, only if the set of its
nonzero words is the set of all subwords of words of some regular
language.

 It is known that every automata algebra can be given by a certain
deterministic graph, and that every finitely defined monomial
algebra is automata (\cite{BBL}, \cite{Ufn}).

The Hilbert series for an automata algebra is rational
(Proposition 5.9 \cite{BBL}). An automata algebra has exponential
growth if and only if $G$ has two cycles $C_1$ and $C_2$ with
common vertex $v$, such that the corresponding words $w_1$, $w_2$
(we read them starting from $v$) are not powers of the same word.
In this case the words $w_1$ and $w_2$ generate a free 2-generated
associative algebra. If there are no such cycles, $A$ has a
polynomial growth. No intermediate growth is possible.

The following theorem is the aim of this section:

\begin{theorem}\label{main}
Let $A=\langle a_1,\dots, a_n\rangle$ be an automata algebra of
exponential growth. Then the Lie algebra $A^{\sim}$  contains a
free $2$-generator subalgebra.
\end{theorem}

We continue to assume that  $G$ is the graph of  a deterministic
FA. Call a se\-mi-re\-gu\-l\-ar word $u$ {\it well--bas\-ed} if it
written on a certain cycle $C$ with an initial vertex $v$; i.e.
$vu=v$. Two se\-mi--reg\-ul\-ar words $u_1$ and $u_2$ are {\it
pa\-ir--wi\-se\-ly well--bas\-ed} if $u_1$ and $u_2$ are written
on cycles $C_1$ and $C_2$ with a common initial vertex $v$ and
$vC_1=vC_2=v$. In this case, for any word $W(a,b)$ the word
$W(u_1,u_2)\ne 0$ in $A$; in particular $u_1^{k_1}u_2^{k_2}\ne 0$.

\medskip
\noindent{\bf Main Lemma.}\ {\it The graph $G$ contains two
regular pairwise well--based words $u\neq v$.}
\medskip

\noindent{\it Deduction of Theorem \ref{main} from the Main
Lemma.} We may always assume that $u\rhd v$. Let $a\succ b$ and
$w$ a regular word. Then (see \cite{BBL}) $w(u,v)$ also is a
regular wold.

For every regular word $u$ we can choose a unique presentation
$u=u_1u_2$ with regular  $u_1$ and regular $u_2$ of maximal
length. In this case $[u]=[[u_1],[u_2]]$ (see \cite{Ufn}).
Therefore, $w(u,v)$ can be obtained by setting $[u]\mapsto a$,
$[v]\mapsto b$ to the word with brackets $[w]$.

Since $u$ and $v$ are well-based, for every word $R(a,b),$ one has
$R(u,v)\neq 0$. Let $[u],[v]$ be the results of the regular
arrangement of the brackets for $u$ and $v$ respectively. Then
$[u]\neq 0,\ [v]\neq 0$.

Thus we have constructed a one-to-one correspondence between the
Hall basis of a Lie algebra, generated by $[u],[v]$ and the Hall
basis of a free $2$-ge\-ne\-ra\-t\-ed Lie algebra with generators
$a,b$. The theorem follows.\ $\Box$

\subsection{Proof of the Main Lemma}
Corollary \ref{Co2SrlrPwrs} implies the following:

\begin{prop}
Suppose the graph $G$ contains two ordered (in the sense of the
operation $\rhd$) pairwise well-based words. Then $G$ contains
also two regular pairwise well-based words.
\end{prop}

It is sufficient to find two ordered semi-regular pairwise
well-based words, i.e. with common final and initial vertices. For
that it is enough to prove the existence of a sufficiently large
number of well-based ordered semi-regular words. In this case,
infinitely many of them will have  a common initial vertex, hence
pairwise well--based, and the main lemma follows.

\begin{lemma}
Let $u_1$ be a well-based word and $u_2$ a cyclically conjugate
word. Then  $u_2$ is also well-based.
\end{lemma}

\begin{proof}
Suppose $u_1=w_1w_2$, $u_2=w_2w_1$ and $v$ is a base vertex of
$u_1$. Then $v'=vw_1$ is a base vertex of $u_2$. Indeed,
$v'u_2=vw_1(w_2w_1)=v(w_1w_2)w_1=vw_1=v'$.
\end{proof}

\begin{cor}
If $u$ is well-based, then  semi-regular word conjugate to $u$ is
also well-based.
\end{cor}

Let $u$ and $d$ be ordered pairwise well-based words (necessarily
not semi-regular). Then, for every $w(a,b),$ the word $w(u,d)$ is
non-zero. In particular, $u^kd^l$ are non-zero for all $k,l$.

Now Lemma \ref{Le2powersConj}, together with the fact that every
non--cyclic word uniquely corresponds to a cyclically conjugated
regular word, implies that there are infinitely many well--based
words. Infinitely many of them will have the same initial vertex
and so will be pairwise well based; hence the Main Lemma follows.

\section{Group theoretical applications}

$\Id(S)$ denotes the ideal, generated by the set $S$.

\begin{lemma}\label{lemma12}
Suppose $a$ and $b$ are homogenous elements of a graded associative
algebra $A$, such that the subalgebra generated by $a,b$ is free
associative algebra with free generators $a,b$. Let $a'$ (resp.
$b'$) be a linear combination of elements in $A$ with degrees
strictly greater than the degree of $a$ (resp. $b$). Let $\tilde
a=a+a'$, $\tilde b=b+b'$. Then the algebra generated by $\tilde a,
\tilde b$ is a free associative with free generators $\tilde a,
\tilde b$.
\end{lemma}

This lemma follows from the fact that for every polynomial
$h(u,v)$ with non-zero minimal component $h'(u,v)$, the minimal
component of $h(\tilde a,\tilde b)$ is $h'(a,b)\neq 0$.

We call an algebra {\it homogenous} if all its defining relations
are homogenous respect to the set of generators. Let $A$ be a
homogenous algebra, and $J$ be an ideal of $A$, generated by
elements of degree $\ge 1$. We call such algebra {\it good}. If
$A/J\equiv k$ and $\bigcap\limits_nJ^n=0$.  Every monomial algebra
is good. For any $x\in A,$ the image of $x$ in $A/J^n$ is not zero
for some $n$, so $A$ can be embedded into the projective limit
$\varprojlim A/J^n$.

\begin{lemma}\label{lemma13}
Let $B$ be a good homogenous algebra, such that $1+a$ and $1+b$
are invertible, $a,b\in J,$ and the elements $a$ and $b$ are free
generators of a free associative subalgebra $C$ of $B$. Then the
group generated by $1+a$ and $1+b$ is free.
\end{lemma}
\vspace{.25cm}

\noindent{\bf Remark.} Note that the pair of two different pairwise
well--bas\-ed words in a monomial algebra generates a free
associative subalgebra.
\medskip

\begin{proof}
Suppose $W(1+a,1+b)=1$ for some non-trivial word $W(x,y)\ne 1$ in
the free group. Consider free algebra $k\langle x,y\rangle$ and
its localization by $1+x,1+y$. Then $W(1+x,1+y)\ne 1$, and, for
some $n_0=n_0(W),$ $W(1+\bar{x},1+\bar{y})\ne 1$ in $\pi(k\langle
x,y\rangle)=k\langle x,y\rangle/\Id(x,y)^n$ for all $n\ge n_0$. In
each such image, the elements $1+x$ and $1+y$ are invertible, so
there is no need for localization.

On the other hand, because $A$ is good homogenous, the image of
$J^m\cap C$ under isomorphism $\phi$ generated by $a\to x, b\to y$
lies in $\Id(x,y)^{n_0},$ and
$$1=\pi(\phi(W(1+a,1+b)))=W(1+\bar{x},1+\bar{y})\ne 1.$$
Contradiction.
\end{proof}

Let $u$ and $v$ be two pairwise well based words. They have
canonical Lie bracket arrangement; let $[u]$ and $[v]$ be
corresponding Lie elements (obtained via opening the Lie
brackets). Notice that
$$
[u]=u+\mbox{lexicographically\ smaller\ terms},\
[v]=v+\mbox{lexicographically\ smaller\ terms}.
$$

Hence we have following

\begin{lemma}\label{lemma13a}
Let $u$ and $v$ be two pairwise well based words, and $[u]$ and
$[v]$ be the corresponding Lie elements (obtained via opening Lie
brackets). Then $[u], [v]$ generate as a free generators
$2$-generated free associative algebra (and also free Lie algebra
via commutator operation).
\end{lemma}

For $n\geq 1$, consider the monomial algebra $A_{n+2}$ with
generators $x_1,\dots, x_{n+2}$ (see the Introduction). Adjoint a
unit $A_{n+2}'=A_{n+2}\cup \{1\}$. The elements $\bar x_i:=1+x_i$
have inverses $\bar x_i^{-1}:=1-x_i$. Consider the group
$A_{n+2}^\#$ generated by the elements $1+x_i$. Consider the
configuration of brackets in the generators of the free subalgebra
in the Lie algebra $A_{n+2}^\sim$ and write the correspondent Lie
elements in the group $A_{n+2}^\#$. Lemmas \ref{lemma12} and
\ref{lemma13}, \ref{lemma13a} imply that the subgroup $A_{n+2}^\#$
generated by these two elements will be free. It is clear from the
construction that all left-normalized commutators of length $k$ in
$A_{n+2}^\#$ which consists of fewer that $k$ different symbols
from $\{\bar x_1,\dots, \bar x_{n+2}\}$, are trivial. Hence,
Theorem \ref{group} follows.

\end{document}